\documentclass[11pt]{article}
\usepackage{bbm}
\usepackage{}
\usepackage{amsfonts}
\usepackage{mathrsfs}
\usepackage{amsmath, amssymb, amsthm, latexsym, amsfonts, epsfig, color,graphicx}

\begin{document}

\newcommand{\s}{\sigma}
\renewcommand{\k}{\kappa}
\newcommand{\p}{\partial}
\newcommand{\D}{\Delta}
\newcommand{\om}{\omega}
\newcommand{\Om}{\Omega}
\renewcommand{\phi}{\varphi}
\newcommand{\e}{\epsilon}
\renewcommand{\a}{\beta}
\renewcommand{\b}{\beta}
\newcommand{\N}{{\mathbb N}}
\newcommand{\R}{{\mathbb R}}
   \newcommand{\eps}{\varepsilon}
   \newcommand{\EX}{{\Bbb{E}}}
   \newcommand{\PX}{{\Bbb{P}}}

\newcommand{\cF}{{\cal F}}
\newcommand{\cG}{{\cal G}}
\newcommand{\cD}{{\cal D}}
\newcommand{\cO}{{\cal O}}

\newcommand{\de}{\delta}

\newcommand{\grad}{\nabla}
\newcommand{\n}{\nabla}
\newcommand{\curl}{\nabla \times}
\newcommand{\dive}{\nabla \cdot}

\newcommand{\ddt}{\frac{d}{dt}}
\newcommand{\la}{{\lambda}}

\newtheorem{theorem}{Theorem}
\newtheorem{lemma}{Lemma}
\newtheorem{remark}{Remark}
\newtheorem{example}{Example}
\newtheorem{definition}{Definition}
\newtheorem{corollary}{Corollary}

\def\proof{\mbox {\it Proof.~}}

\title{Stochastic Averaging Principle for Dynamical Systems with
Fractional Brownian Motion}

\author{Yong Xu$^1$, Rong Guo$^1$, Di Liu$^1$, Huiqing Zhang$^1$, Jinqiao Duan$^2$
\vspace{1mm}\vspace{1mm}\\
{\it\small 1.   Department of Applied Mathematics}\\
{\it\small  Northwestern Polytechnical University, }
 {\it\small
Xi'an, 710072, China }\\
{\it\small E-mail:  hsux3@nwpu.edu.cn; huiqingzhang@nwpu.edu.cn} \vspace{1mm}\\
{\it\small 2. Department of Applied Mathematics}\\
 {\it\small Illinois Institute of Technology},
 {\it\small Chicago, IL 60616, USA}\\
{\it\small E-mail: duan@iit.edu}  \vspace{1mm} }

\date{\today  }

\maketitle

\begin{abstract}

\noindent  Stochastic averaging for a class of stochastic differential equations (SDEs) with fractional Brownian motion, of the Hurst parameter $H$ in the interval $(\frac{1}{2},1)$, is investigated. An averaged SDE for the original SDE  is proposed, and their solutions are quantitatively compared. It is shown that the solution  of the averaged SDE converges to that of the original SDE in the sense of mean square and also in probability. It is further demonstrated that a similar averaging principle holds   for   SDEs under   stochastic integral of pathwise backward and forward types. Two examples are presented and numerical simulations are carried
out to illustrate the  averaging principle.

\end{abstract}
\noindent {\it \footnotesize Key words}. {\scriptsize correlated noise; Averaging
principle; Stochastic differential equations; \\ Stochastic calculus; Fractional Brownian motion.}

\noindent {\it \footnotesize 2000 Mathematics Subject
Classification}. {\scriptsize Primary: 34F05, 37H10, 60H10,
93E03}.




\setcounter{secnumdepth}{5} \setcounter{tocdepth}{5}

\makeatletter
    \newcommand\figcaption{\def\@captype{figure}\caption}
    \newcommand\tabcaption{\def\@captype{table}\caption}
\makeatother


\section{\bf Introduction}

Stochastic averaging is often   used  to   approximate
   dynamical systems under random fluctuations. This analytic technique  has been developed in the case of the Gaussian random fluctuations, for example, by Stratonovich \cite{Stratonovich1, Stratonovich2}  and then
by Khasminskii \cite{Khasminskii1, Khasminskii2}. It has been found to be effective
for understanding stochastic differential equations arising in many fields \cite{Namachchivaya, Roberts, Liptser}.  Zhu and his co-workers further studied this stochastic averaging method   for nonlinear systems under Poisson   noise  \cite{Zhu,Zeng,Jia}, and two of the present authors derived an averaging principle  for stochastic
differential equations with L\'{e}vy noise \cite{Xu}. In all of these mentioned works, the fluctuations or noises are uncorrelated, i.e., white noises.
\\
\mbox{}\quad However,  random fluctuations with long-range dependence,  or correlated noises, are abundant. They may be   modeled by fractional Brownian motion (fBm) with $ \frac{1}{2}< H < 1 $ (where $H$ is the Hurst index). The fractional Brownian motion was
introduced  by Kolmogorov  \cite{Kolmogorov}. Then, in 1968, Mandelbrot and Van Ness  \cite{Mandelbrot}  presented the structure of the fractional Brownian motion.
Due to the importance of long-range dependence of the fBm, the stochastic differential equations with fBm  have been used as the model of the practical problems in various fields, such as hydrology, queueing theory and mathematical finance (Chakravarti and Sebastian, \cite{Chakravarti}; Hu and {\O}ksendal, \cite{Hu};
Leland, Taqqu, Willinger, and Wilson et al, \cite{Leland}; Scheffer, \cite{Scheffer}). So fractional Brownian motion has also been suggested as a replacement of standard Brownian motion in several stochastic models (\cite{Norros, Shiryaev, Baillie}).\\

\mbox{}\quad Given the abundance of correlated fluctuations, it is crucial to understand  the behaviors of the stochastic differential equations with fBm \cite{Jumarie1, Jumarie2, Kou, Sliusarenko}.  Unfortunately, the fractional Brownian motion is neither a semi-martingale nor a Markov process, so the powerful tools for the stochastic integral theories are not applicable
when studying fBm. Therefore, much of the recent research on SDEs with fBm is by numerical simulations. Other techniques for such SDEs would be desirable. This motivates us to investigate   stochastic averaging techniques for differential equations
driven by fractional Brownian motion. \\

In the present paper, we   study a stochastic averaging technique for a class of  SDEs with fBm.   We present an averaging principle, and prove that the original stochastic differential equation  can be approximated by an averaged stochastic differential equation  in the sense of mean square convergence and convergence in probability, when a scaling parameter tends to zero. In addition, the similar conclusion holds for a SDE, where the stochastic differential or stochastic integral is of forward and backward types.


\mbox{}\quad The organization of the paper is as follows. Section 2 recalls the
definition of fractional Brownian motion  and highlight the
differences with the usual Brownian motion roughly, and then briefly reviews the symmetric, forward and backward stochastic integrals with respect to fBm.
Section 3 is devoted to prove a stochastic averaging principle for stochastic differential equations
  with fBm.   Section 4   presents two examples to illustrate the stochastic averaging principle.


\section{Fractional Brownian motion and stochastic integration}

Since stochastic differential equations are interpreted via stochastic integrals, it is necessary to specify the integration with respect to fBm. For background on this issue, see  \cite{Dai, Decreusefond, Feyel, Carmona, Nualart1}. For instance, using the notions of fractional integral and derivative, it is appropriate to  introduce  a pathwise stochastic integral with respect to fBm  \cite{Young, Zahle, Russo}.


In this preliminary section, we  briefly recall the definition of fBm  and   the integration with respect to it, for $H \in (\frac12, 1)$.

\subsection{Fractional Brownian motion}
Let $(\Omega,\mathscr{F},P)$ be a complete probability space. The definition of the fractional Brownian motion is as follows\cite{Mandelbrot}.

\begin{definition}
The fractional Brownian motion ($B^H(t)$) with Hurst index $H$ is a centered self-similar Gaussian process
$ B^H={B^H(t), t\in \mathbb{R_+} }$, on $ (\Omega,\mathscr{F},P) $ with the properties :\\
 (1)  $ B^H(0)=0 $ ;\\
 (2)  $ \mathbb{E}{B^H(t)}=0, t \in \mathbb{R_+} $ ;\\
 (3)  $ \mathbb{E}{B^H(t)B^H(s)} = \frac{1}{2}(|t|^{2H} + |s|^{2H} - |t-s|^{2H}), t, s \in \mathbb{R_+} $.\\
\end{definition}
For $ H = \frac{1}{2} $, this is the usual Brownian motion.

We also recall the following features of the fractional Brownian motion:\\

{\bf (a)} Self-similarity : For every constant $ a > 0 $ and every $ T > 0 $,   the following relation about distribution (or law) holds
\begin{eqnarray*}
       Law\quad (B^H(at):t \in [0,T]) = Law\quad (a^H B^H(t):t \in [0,T]).
\end{eqnarray*}
The above formula means that the two processes $ B^H(at) $ and $ a^H B^H(t) $ have the same finite-dimensional distribution functions, i.e., for
every choice of $ t_0, \ldots, t_n \in \mathbb{R_+} $,
\begin{eqnarray*}
 P(B^H(at_0) \leq x_0, \ldots, B^H(at_n) \leq x_n) = P(a^H B^H(t_0) \leq x_0, \ldots, a^H B^H(t_n) \leq x_n),
\end{eqnarray*}
 for every $ x_0, \ldots, x_n \in \mathbb{R} $.\\

{\bf (b)} Stationary increments : The increment of this process in  $ (s,t) $ has a normal distribution with zero mean, and the following variance
\begin{eqnarray*}
 \mathbb{E}{(B^H(t)-B^H(s))^2} = |t-s|^{2H}.
\end{eqnarray*}
Hence, for every integer $ k \geq 1 $,
\begin{eqnarray*}
 \mathbb{E}{(B^H(t)-B^H(s))^{2k}} = \frac{2k!}{k!2^k}|t-s|^{2Hk}.
\end{eqnarray*}
In other words, the parameter $H$ controls the regularity of the trajectories. For $ H = \frac{1}{2} $, the increments of the process in disjoint intervals are
independent, while for $ H \neq \frac{1}{2} $, the increments  are dependent.\\

{\bf (c)} Long-range dependence : The auto-covariance function  $ \rho_H(n),n \in N $ of the fBm is
\begin{eqnarray*}
\quad\quad \quad \rho_H(n): =Cov(B^H (k)-B^H(k-1),B^H(k+n)-B^H(k+n-1))\\
= \frac{1}{2}[(n+1)^{2H}+(n-1)^{2H}-2n^{2H}]\\
\approx H (2H-1)n^{2H-2},
\end{eqnarray*}
and
$ \rho_H(n) \longrightarrow 0 $, as n tends to infinity.\\
\mbox{}\quad If $ H > \frac{1}{2}, \;  \rho_H(n) > 0 $, for $n$ large enough,  and $ \sum_{n=1}^\infty \rho_H(n) = \infty $. In this case, we say that the fractional Brownian motion has long-range dependence.
So the fBm can be used to describe cluster phenomena, occuring in geophysics, hydrology and economics.\\
\mbox{}\quad Based on the definition of the fractional Brownian motion, it is clear  that the standard Brownian motion is a specific fractional Brownian motion with index $ H = 1/2 $. \\
\mbox{}\quad The relationship between the usual Brownian motion and fractional Brownian motion is as follows:\\
{(R1)} The similarities :
They are both Gaussian process; they do not have differentiable sample paths and both have statistical self-similarity; besides they are almost everywhere H\"{o}lder continuous.\\
{(R2)} The differences : Fractional Brownian motion is neither a semi-martingale nor a Markov process (for $ H \neq \frac{1}{2}$), but the usual Brownian motion is a semi-martingale and a Markov process; fractional Brownian motion has no independent increments, while the usual Brownian motion has.

\subsection{Stochastic integration  with  respect  to     fractional Brownian motion}

For the convenience of   readers, we recall   some  stochastic integration with respect to the fractional Brownian motion   \cite{Biagini, Duncan, Mishura}.\\
Let $ \phi : \mathbb{R}_{+} \times \mathbb{R}_{+}  \longrightarrow \mathbb{R}_{+} $ be given by\\
\begin{eqnarray*}
 \phi(t,s) = H (2H-1)|t-s|^{2H-2}, t,s \in \mathbb{R}_{+},
\end{eqnarray*}\\
where $ \frac{1}{2}<H<1 $, and
let $ f : \mathbb{R}_{+} \longrightarrow \mathbb{R}_{+} $  be Borel measurable.  Define
\begin{eqnarray*}
L_\phi^2=\{f:|f|_\phi^2  = \int_\mathbb{R} \int_\mathbb{R} f(t)f(s) \phi (t,s) dsdt < \infty \}.
\end{eqnarray*}\\
The Hilbert space $ L_\phi^2 $ is naturally associated with the Gaussian process $ (B^H (t),t\geq 0)$.\\
Let $\mathcal {S}$ be the set of smooth and cylindrical random variables of the form
\begin{eqnarray*}
F=f(B^H (\psi_1),B^H (\psi_2),\ldots,B^H(\psi_n)),
\end{eqnarray*}
where $ n \geq 1$, $ f \in \mathcal {C}_b^\infty (\mathbb{R}^n) $ (i.e., $f$ and all its partial derivatives are bounded), and $\psi_i \in \mathcal {H}$, $\mathcal {H}$ is a Hilbert space \cite{Nualart1}.\\\
Introduce    the Malliavin $ \phi-derivative $ of $ F$
\begin{eqnarray*}
D_t^\phi F = \int_\mathbb{R} \phi(t,\nu) D^H F d\nu,
\end{eqnarray*}
where
\begin{eqnarray*}
D^H F=\sum_{i=1}^n \frac{\partial f}{\partial x_i} (B^H(\psi_1),\ldots,B^H(\psi_n)) \psi_i.
\end{eqnarray*}

In this paper, we consider the pathwise stochastic integrals for fBm. The definition of the symmetric stochastic integral for the fBm case is in \cite{Biagini}.
\begin{definition}
\label{definition}
Let $ u(t)$ $(t\in [0,T])$ be a stochastic process with integrable trajectories. The symmetric integral of $ u(t) $ with respect to $ B^H(t)$ is defined as
\begin{eqnarray*}
\lim_{\epsilon \rightarrow 0} \frac{1}{2\epsilon}\int_0^T u(s) [B^H (s+\epsilon )-B^H (s-\epsilon)]ds,
\end{eqnarray*}
provided that the limit exists in probability, and is denoted by $ \int_0^T u(s)d^{\circ}B^H(s)$.
\end{definition}

\begin{remark}
Let $L(0,T)$ be the family of processes on $[0,T]$, such that $u(t) \in L(0,T)$ if  $\mathbb{E}|u(t)|_\varphi^2 <\infty$. Assume that $u(t)$ $(t\in [0,T])$ is a stochastic process  in $L(0,T)$ and satisfies
\begin{eqnarray*}
\int_0^T \int_0^T |D_s^H u(t)||t-s|^{2H-2}dsdt <\infty.
\end{eqnarray*}
Then the symmetric integral exists and the following relation holds:
\begin{equation}
\int_0^T u(t)d^\circ B^H(t)= \int_0^T u(t)\diamond dB^H(t) + \int_0^T D_s^\varphi u(s)ds,
\end{equation}
where $\diamond$ denotes the Wick product, $ H\in(\frac{1}{2},1)$.
\end{remark}

\begin{remark}
The definition of the forward and backward integrals with respect to  fBm is as follows:\\
 Let $u(t)$ $(t\in [0,T])$ be a process with integrable trajectories. 
The forward integral of $ u(t) $ with respect to $ B^H(t)$ is defined as
\begin{eqnarray*}
\lim_{\epsilon \rightarrow 0} \frac{1}{\epsilon}\int_0^T u(s) [\frac{B^H (s+\epsilon )-B^H (s)}{\epsilon}]ds,
\end{eqnarray*}
provided that the limit exists in probability, and is denoted by $ \int_0^T u(s)d^{-}B^H(s)$.

The backward integral is defined as
\begin{eqnarray*}
\lim_{\epsilon \rightarrow 0} \frac{1}{\epsilon}\int_0^T u(s) [\frac{B^H (s-\epsilon )-B^H (s)}{\epsilon}]ds,
\end{eqnarray*}
provided that the limit exists in probability, and is denoted by $ \int_0^T u(s)d^{+}B^H(s)$.
\end{remark}

\begin{remark}
According to \cite{Biagini}, under the assumptions in Remark $1$, the symmetric, backward and forward integrals coincide in the following sense
\begin{equation}
\int_0^T u(t)d^- B^H(t)= \int_0^T u(t)\diamond dB^H(t) + \int_0^T D_s^\varphi u(s)ds,
\end{equation}
\begin{equation}
\int_0^T u(t)d^+ B^H(t)= \int_0^T u(t)\diamond dB^H(t) + \int_0^T D_s^\varphi u(s)ds.
\end{equation}
\end{remark}

\section{\bf An averaging principle for   SDEs   with fBm}
\subsection{ Some Lemmas }

In order to present a stochastic averaging principle, we   need two lemmas.\\
\begin{lemma}
\label{lemma:}
Let $ B^H(t) $ be a fractional Brownian motion with $ \frac{1}{2}<H<1$, and $ Z(s) $ be a stochastic process in $ L[0,T]$ .
For every $ T < \infty $ , there exists a  constant $ C(H,T)= H T^{2H-1} $ such that
\begin{equation}
\mathbb{E}[(\int_0^T |Z(s)|\diamond dB^H(s))^2] \leq C(H,T)\mathbb{E}[\int_0^T |Z(s)|^2 ds] + C T^2.
\end{equation}
\end{lemma}
\begin{proof}
According to \cite{Duncan} ( $ Theorem$  $2.1$  ),  \\
\begin{eqnarray*}
\mathbb{E}{\int_0^T (D_s^\varphi |Z(s)|)^2 ds }<\infty,
\end{eqnarray*}
and
\begin{eqnarray*}
 \mathbb{E}{(\int_0^T |Z(s)|\diamond dB^H(s))^2} = \mathbb{E}{[||Z||_\varphi^2 + \int_0^T D_s^\phi |Z(s)| ds)^2]}.
\end{eqnarray*}
Thus,
\begin{eqnarray*}
& & \mathbb{E}{(\int_0^T |Z(s)|\diamond dB^H(s))^2} = \mathbb{E}{[\int_0^T \int_0^T Z(t)Z(s)\phi(t,s) ds dt} + \mathbb{E}{(\int_0^T D_s^\phi Z(s) ds)^2}\\
& & \quad \quad \quad \quad \quad \quad \quad \quad \quad
\quad\quad \triangleq  A+B,
\end{eqnarray*}
where\\
\begin{eqnarray*}
& \quad \quad \quad \quad A = \mathbb{E}{\int_0^T \int_0^T |Z(t)||Z(s)| \phi(t,s) ds dt}, \\
& B = \mathbb{E}{(\int_0^T D_s^\phi |Z(s)| ds )^2}.
\end{eqnarray*}
We further have
\begin{eqnarray*}
A \leq  H (2H-1) \mathbb{E}{\int_0^T \int_0^T |Z(s)|^2|t-s|^{2H-2} ds dt}\\
\leq  H (2H-1) \mathbb{E}{\int_0^T |Z(s)|^2 \frac{1}{2H-1} T^{2H-1}ds}\\
\leq  H T^{2H-1} \mathbb{E}{\int_0^T |Z(s)|^2}ds \\
\leq C(H,T) \mathbb{E}{\int_0^T |Z(s)|^2 ds}.
\end{eqnarray*}
By the Cauchy-Schwarz inequality for $ B $ , we get :
\begin{eqnarray*}
B \leq  T^2 \mathbb{E}{\int_0^T |D_s^\phi |Z(s)||^2 ds} \leq C T^2.
\end{eqnarray*}
Then we can finally deduce that
\begin{eqnarray*}
\mathbb{E}[(\int_0^T |Z(s)|\diamond dB^H(s))^2] \leq  C(H,T) \mathbb{E}{\int_0^T |Z(s)|^2 ds} + C T^2.
\end{eqnarray*}
This finishes the proof of this Lemma.\\
\end{proof}

Lemma 2 can be obtained according to Definition 4 and Lemma 1.

\begin{lemma}
\label{lemma}
 Suppose that Z(s) is a stochastic process in $L[0,T]$, and $B^H(t)(H>\frac{1}{2})$
is a fractional Brownian motion. For any $0<T < \infty $, there
exists a constant $ C(H,T) $, such that the following inequality holds
\begin{equation}
\mathbb{E}[(\int_0^T |Z(s)| d^\circ B^H(s))^2] \leq 2C(H,T)\mathbb{E}[\int_0^T |Z(s)|^2 ds] + 4 C T^2,
\end{equation}
where $ C(H,T) = H T^{2H-1} $ .
\end{lemma}
\begin{proof}
Using Eq.(1) and the Cauchy-Schwarz
inequality, we can get
\begin{eqnarray*}
\mathbb{E}[(\int_0^T |Z(s)|d^\circ B^H(s))^2] =
\mathbb{E}[(\int_0^T |Z(s)|\diamond dB^H(s) + \int_0^T D_s^\varphi |Z(s)|ds)^2]\\
 \leq \mathbb{E}[2(\int_0^T |Z(s)|\diamond B^H(s))^2 + 2(\int_0^T D_s^\varphi |Z(s)|ds)^2]\\
 \leq 2\mathbb{E}[(\int_0^T |Z(s)|\diamond B^H(s))^2] +
2\mathbb{E}[\int_0^T D_s^\varphi |Z(s)|ds^2],\\
\end{eqnarray*}
Due to Eq.(4) and
\begin{eqnarray*}
\mathbb{E}{(\int_0^T D_s^\phi |Z(s)| ds )^2}\leq C T^2,
\end{eqnarray*}
we obtain
\begin{eqnarray*}
 \mathbb{E}[(\int_0^T |Z(s)| d^\circ B^H(s))^2]
\leq 2 C(H,T)\mathbb{E}{\int_0^T |Z(s)|^2 ds} + 2C T^2
+ 2 C T^2,
\end{eqnarray*}
namely
\begin{eqnarray*}
\mathbb{E}[(\int_0^T |Z(s)| d^\circ B^H(s))^2]\leq 2C(H,T)\mathbb{E}[\int_0^T |Z(s)|^2 ds] + 4C T^2.
\end{eqnarray*}
the proof is completed.
\end{proof}

\subsection{Stochastic differential equations driven by fractional Brownian motion }

In this section, we concern the symmetric integral of stochastic differential equations with respect to fBm. Solutions of the stochastic differential equation driven by fractional Browinan motion have been studied intensively by using the pathwise approach \cite{Lyons,Nualart2}.\\
Consider the equation on $\mathbb{R}^d$\\
\begin{equation}
X(t)=X(0)+\int_0^t b(s,X(s))ds + \int_0^t \sigma(s,X(s))d^\circ B^{H} (s),
\end{equation}
where $X(0)$ is a given $d$-dimensional random variable, $b(s,X(s)):[0,T] \times \mathbb{R}^d \longrightarrow \mathbb{R}^d$ is a measurable vector function, $\sigma(s,X(s))is a d\times m $ matrix with each element $\sigma_{j,i}(s,X(s)): [0,T] \times \mathbb{R}^d \longrightarrow \mathbb{R}^d $ a measurable vector function, and the processes $B^{H}(t)$, represents $d$-dimensional fractional Brownian motions with Hurst parameter $H$ defined in a complete probability space $(\Omega,\mathcal {F},\mathbb{P})$. Denote by $ \sigma = (\sigma_{j,i})_{i,j=1}^{d,m}$ the matrix of "diffusion" and $ b= (b_i)_{i=1}^d $ the "drift" vector, $ |\sigma| :=(\sum_{i,j} |\sigma_{j,i}|^2)^\frac{1}{2} $ , $ |b|:= (\sum_i (b_i)^2)^\frac{1}{2}$ .\\

Let us consider the following assumptions on the coefficients :\\

$ \sigma (t,x)$ is differentiable in $ x$ , and satisfies :
there exists $ M>0 $ , $ 0 < \gamma $ , $ k \leq 1 $ , and for any $ N > 0 $,  $ M_N > 0 $,
\\
{\bf (i)} $ \sigma $ is Lipschitz continous in $ x $ ,$ \forall$ $ x $, $ y   \in \mathbb{R}^d $ , $ t \in [0,T] $ :
\begin{eqnarray*}
    |\sigma(t,x)-\sigma(t,y)| \leq  M |x-y|,
\end{eqnarray*}
{\bf (ii)} $ x $ -derivative of $ \sigma $ is local H\"{o}lder continous in $ x $ , $ \forall $ $ |x|,|y| \leq  N , t \in [0,T]$
\begin{eqnarray*}
    |\sigma_{x}(t,x)-\sigma_{x}(t,y)| \leq  M_N |x-y|^k,
\end{eqnarray*}
{\bf (iii)}  $ \sigma $ is H\"{o}lder continous in times, for all $ x \in \mathbb{R}^d , t, s \in [0,T]$ :
\begin{eqnarray*}
    |\sigma(t,x)-\sigma(s,x)| + |\sigma_{x_i}(t,x)-\sigma_{x_i}(s,x)| \leq  M|t-s|^\gamma.
\end{eqnarray*}
for each $ i=0,\ldots , d$.\\

The function $ b=b(t,x)$ satisfies the following conditions:\\
{\bf (iv)} for all $ N \geq 0 $ , there exist $ L_N > 0 $ , for all $ |x|,|y| \leq N , t \in [0,T] $ , such that
\begin{eqnarray*}
   |b(t,x)-b(t,y)| \leq L_N|x-y|,
\end{eqnarray*}
{\bf (v)} there exists the function $ b_0 \in L_p(0,T;\mathbb{R}^d) $ $(p\geq 2)$ , and $ L > 0 $ , for any $ (t,x) \in [0,T] \times \mathbb{R} $ such that
\begin{eqnarray*}
   |b(t,x)| \leq L|x| + b_0(t).
\end{eqnarray*}\\
On the basis of $ Theorem $  $ 3.1.4 $  and $ Remark $ $3.1.5$ in \cite{Mishura}, there exists the unique solution $(X_t,t\in [0,T])$ of the Eq.(6).

\subsection{\bf An averaging principle}
Now we discuss a standard stochastic differential equation using an averaging principle in $ \mathbb{R}^d $.\\
\mbox{}\quad The standard stochastic differential equation is defined as:\\
\begin{equation}
\begin{aligned}
    X_\epsilon(t) = X(0) + \epsilon^{2H} \int_0^t b(s, X_\epsilon(s)) ds
    + \epsilon^{H}  \int_0^t \sigma(s, X_\epsilon(s))d^\circ B^H(s).
\end{aligned}
\end{equation}
where $ X(0) = X_0 $ is a given $d$-dimensional random varibale as the initial condition, $t\in [0, T]$ and the coefficients have the
same conditions as in Eq.$(6)$, and $ \epsilon \in
(0,\epsilon_0]$ is a positive small parameter with $ \epsilon_0 $ a
fixed number.\\
\mbox{}\quad Assume that $(i)-(v)$ (the Lipschitz and growth conditions) are satisfied, besides the mappings $ \bar{b}: \mathbb{R}^d \rightarrow
\mathbb{R}^d $ , $ \bar{\sigma}: \mathbb{R}^d \rightarrow \mathbb{R}^d $ ,
are measurable. And presume they meet the following additional inequalities :
\\
{\bf (C1)}
\begin{eqnarray*}
   \frac{1}{T_1} \int_0^{T_1} | b(s,y)- \bar {b}(y)| ds
   \leq \varphi_1(T_1)(1+ |y|),
\end{eqnarray*}\\
{\bf (C2)}
\begin{eqnarray*}
   \frac{1}{T_1} \int_0^{T_1} |\sigma(s,y)- \bar{\sigma}(y)|^2
   ds \leq \varphi_2(T_1)(1+ | y|^2).
\end{eqnarray*}
where $ T_1 \in [0,T],  \varphi_i(T_1)$ are positive bounded functions
with $ \lim_{T_1 \rightarrow \infty} \varphi_i(T_1)$ \\$= 0$, $i = 1,2$.

Then, we can obtain the SDEs with the averaging principle :
\begin{equation}
    Z_\epsilon(t) = X(0) + \epsilon^{2H} \int_0^t \bar{b}(Z_\epsilon(s)) ds + \epsilon^{H} \int_0^t \bar{\sigma}
    (Z_\epsilon(s))
    d^\circ B^H(s).
\end{equation}
This SDE is called the averaged SDE of the original standard SDE (7). Under the similar conditions such as
$ X(t)$ in Eq.(6), this equation will have a unique solution
$ Z_\epsilon(t) $.\\
\mbox{}\quad Now We claim the following main theorems to show relationship between solution
processes $ X_\epsilon(t) $ and $ Z_\epsilon(t) $.
\medskip
It shows that the solution of averaged Eq.(8) converges to that
of the original Eq.(7) in the sense of mean square and probability respectively. \\

\begin{theorem}
Suppose that the original SDEs (7) and the averaged SDEs (8) both satisfy the assumptions
  (i)-(v)and (C1)-(C2). For a given arbitrarily small
number $ \delta_1 > 0 $ , there exist $ L > 0 $ , $
\epsilon_1 \in (0,\epsilon_0] $ and  $ \beta \in (0, 1) $,
such that for any $ \epsilon \in (0,\epsilon_1] $ ,
\begin{equation*}
     \mathbb{E}(|X_\epsilon(t)- Z_\epsilon(t)|^2)\leq \delta_1.
\end{equation*}
\end{theorem}

\begin{remark}
{\bf (i)} This conclusion shows that the solution of
averaged SDEs converges to that of initial SDEs in a certain sense. That is {\bf\textit{Theorem 1}} means the convergence of these two
solutions in the sense of mean square.

 { \bf (ii)} If only partial conditions hold,   {\bf\textit{Theorem 1}} may still hold.
 In this situation we may speak of partial averaging.

\end{remark}
\begin{proof}

According to the above analysis, we start with

\begin{eqnarray*}
 X_\epsilon(t)- Z_\epsilon(t) = \epsilon^{2H} \int_0^t
[b(s,X_\epsilon(s)) -\bar{b}(Z_\epsilon(s))] ds  + \\
\quad \quad\quad \epsilon^{H} \int_0^t [\sigma
    (s,X_\epsilon(s)) - \bar{\sigma}
    (Z_\epsilon(s))]
    d^\circ B^H(s),
\end{eqnarray*}
and employ the following inequality for $ n \in \mathbb{N} $, and $ x_1 , x_2 , \ldots , x_n \in \mathbb{R} $ :
\begin{eqnarray}
   |x_1 + x_2 + \ldots + x_m|^2 \leq m (|x_1|^2 + |x_2|^2 + \ldots + |x_m|^2),
\end{eqnarray}

we arrive at
\begin{eqnarray*}
& & |X_\epsilon(t)- Z_\epsilon(t)|^2 \leq 2
 \epsilon^{4H} | \int_0^t
[b(s,X_\epsilon(s)) - \bar{b}(Z_\epsilon(s))] ds|^2 + \\
& &　\quad\quad\quad \quad\quad \quad\quad\quad 2 \epsilon^{2H} |\int_0^t [\sigma
    (s,X_\epsilon(s))-\bar{\sigma}
    (Z_\epsilon(s))]
    d^\circ B^H(s)|^2\\
& &　\quad\quad \quad\quad\quad\quad\quad  = I_1^2 + I_2^2.
\end{eqnarray*}
where  $ [0,t]\in [0,u]\in [0,T]$,$ I_i , i=1 , 2 $ denote the above terms
respectively. Now we present some estimates for
$ I_i , i=1 , 2 $.

Firstly, we apply the inequality (9) to get
\begin{eqnarray*}
& & I_1^2 = 2\epsilon^{4H}|\int_0^t
[b(s,X_\epsilon(s)) - \bar{b}(Z_\epsilon(s))] ds|^2  \\
& & \quad  \leq 2\epsilon^{4H} |\int_0^t
[b(s,X_\epsilon(s))-b(s,Z_\epsilon(s))+ b(s,Z_\epsilon(s))
-\bar{b}(Z_\epsilon(s))] ds|^2 \\
& & \quad  \leq 4\epsilon^{4H} | \int_0^t
[b(s,X_\epsilon(s))-b(s,Z_\epsilon(s))] ds |^2 +\\
& &  \quad\quad 4\epsilon^{4H} |\int_0^t
[b(s,Z_\epsilon(s)) -\bar{b}(Z_\epsilon(s))] ds|^2\\
& & \quad  \leq I_{11}^2 + I_{12}^2,
\end{eqnarray*}
 where
\begin{eqnarray*}
 & & I_{11}^2 = 4\epsilon^{4H} | \int_0^t
[b(s,X_\epsilon(s))- b(s,Z_\epsilon(s))]ds|^2,  \\
& & I_{12}^2 = 4\epsilon^{4H}  |\int_0^t
[b(s,Z_\epsilon(s)) - \bar{b}(Z_\epsilon(s))]ds|^2.
\end{eqnarray*}

By the Cauchy-Schwarz inequality for $ I_{11}^2 $ , we obtain :
\begin{eqnarray*}
 |I_{11}|^2 \leq 4\epsilon^{4H}  t |\int_0^t
[b(s,X_\epsilon(s))- b(s,Z_\epsilon(s))]^2 ds|,
\end{eqnarray*}

Because of condition (ii) and taking expectation, we can get
\begin{eqnarray*}
 & &  \mathbb{E}|I_{11}|^2 \leq 4\epsilon^{4H} \mathbb{E} ( t \int_0^t
|b(s,X_\epsilon(s))- b(s,Z_\epsilon(s))|^2ds )\\
& &\quad \quad \quad  \leq 4\epsilon^{4H} u L_N^2 \mathbb{E}
 \int_0^t |X_\epsilon(s)-Z_\epsilon(s)|^2 ds)\\
& &  \quad \quad \quad  \leq 4\epsilon^{4H} u L_N^2 \mathbb{E} (\int_0^u |X_\epsilon(s)-Z_\epsilon(s)|^2 ds)\\
& &  \quad \quad \quad  \leq 4\epsilon^{4H} u K_{11} \int_0^u
\mathbb{E} (
|X_\epsilon(s_1)-Z_\epsilon(s_1)|^2 ) du.
\end{eqnarray*}
where $ K_{11} $ is a constant.

Then about $ I_{12}^2 $ , we use condition (C1), $ \varphi_1(t)$ is positive bounded function and take expectation to
yield :
\begin{eqnarray*}
 & & \quad\mathbb{E} |I_{12}|^2 \leq  4\epsilon^{4H} \mathbb{E}( t^2
[\frac{1}{t}\int_0^t
|b(s,Z_\epsilon(s)) -\bar{b}(Z_\epsilon(s))|ds]^2)\\
 \\& & \quad \quad \quad \quad \leq 8\epsilon^{4H} u^2 \mathbb{E} \{ \varphi_1(t)^2 (1+
|Z_\epsilon(s)|^2))\}\\
& & \quad\quad\quad \quad \leq 8\epsilon^{4H} u^2 K_{12} (1+
\mathbb{E}(\sup_{0 \leq t \leq u}(|Z_\epsilon(t)|^2))\\
& &  \quad\quad\quad\quad \leq 8\epsilon^{4H} u^2 K_{12},
\end{eqnarray*}
where $ K_{12} $ denotes a constant which may differ in the above
inequality. For each $ t \geq 0 $ , we get
\begin{eqnarray*}
 \mathbb{E}|I_{1}|^2 \leq 4\epsilon^{4H} u K_{11} \int_0^u
\mathbb{E} (|X_\epsilon(s_1)-Z_\epsilon(s_1)|^2 ) du+ 8\epsilon^{4H} u^2 K_{12}.      \quad\quad        (Z_1)
\end{eqnarray*}
Now take expectation on $ I_2^2 $ to obtain
\begin{eqnarray*}
 & & \mathbb{E} |I_{2}|^2 = 2 \epsilon^{2H}  \mathbb{E} |\int_0^t [\sigma
    (s,X_\epsilon(s))-\bar{\sigma}
    (Z_\epsilon(s))]
    d^\circ B^H(s))^2|\\
 & &\quad \quad \quad  \leq 4 \epsilon^{2H} \mathbb{E} (|\int_0^t [\sigma
    (s,X_\epsilon(s))-\sigma
    (s,Z_\epsilon(s))]d^\circ B^H(s)|^2)+\\
 & & \quad \quad \quad  \quad 4 \epsilon^{2H} \mathbb{E} (|\int_0^t [\sigma
    (s,Z_\epsilon(s))-\bar{\sigma}
    (Z_\epsilon(s))]d^\circ B^H(s)|^2) \\
& & \quad\quad\quad = I_{21}^2+I_{22}^2.
\end{eqnarray*}
where
\begin{eqnarray*}
& & I_{21}^2= 4 \epsilon^{2H}
\mathbb{E} ( |\int_0^t [\sigma
    (s,X_\epsilon(s))-\sigma
    (s,Z_\epsilon(s))]d^\circ B^H(s)|^2), \\
& & I_{22}^2=4 \epsilon^{2H} \mathbb{E} (|\int_0^t [\sigma
    (s,Z_\epsilon(s))-\bar{\sigma}
    (Z_\epsilon(s))]d^\circ B^H(s)|^2).
\end{eqnarray*}

  By the Lemma 2, conditions (i) and (C2), it is easy to get
\begin{eqnarray*}
 & & \quad I_{21}^2 \leq 4 \epsilon^{2H} \mathbb{E}(|\int_0^t
 M |X_\epsilon(s)-Z_\epsilon(s)|d^\circ B^H(s)|^2)\\
 & & \quad\quad\quad \leq 4 \epsilon^{2H} M^2 \mathbb{E}
 |\int_0^t |X_\epsilon(s)-Z_\epsilon(s)|d^\circ B^H(s)|^2\\
 & & \quad\quad\quad \leq 4 \epsilon^{2H} M^2 (2H t^{2H-1} \mathbb{E}
 [\int_0^t|X_\epsilon(s)-Z_\epsilon(s)|^2ds]+4C t^2)\\
 & & \quad\quad\quad \leq 8 \epsilon^{2H} H u^{2H-1} K_{211}\int_0^u \mathbb{E}(|X_\epsilon(s_1)-Z_\epsilon(s_1)|^2)du + 16\epsilon^{2H} u^2 K_{212}.
\end{eqnarray*}

 Due to the conditions (C2) , we obtain
\begin{eqnarray*}
 & & \quad I_{22}^2\leq 4 \epsilon^{2H} \mathbb{E}(|\int_0^u
 [\sigma(s,Z_\epsilon(s))-\bar{\sigma}(Z_\epsilon(s))]d^\circ B^H(s)|^2)\\
 & & \quad\quad\quad\leq 4\epsilon^{2H} \{2H u^{2H-1} \mathbb{E}
 [\int_0^t|\sigma(s,Z_\epsilon(s)-\bar{\sigma}(Z_\epsilon(s))|^2ds]+ 4 C u^2\}\\
 & & \quad\quad\quad\leq 8\epsilon^{2H} H u^{2H-1} \mathbb{E}[\int_0^t |\sigma(s,Z_\epsilon(s)-\bar{\sigma}(Z_\epsilon(s))|^2ds]+
 16\epsilon^{2H}Cu^2\\
 & & \quad\quad\quad\leq 8\epsilon^{2H} H u^{2H} K_{221}\{ [1+\mathbb{E}(\sup_{0 \leq t \leq u}
 |Z_\epsilon (t)|^2)]\}+16 \epsilon^{2H} C u^2 \\
 & & \quad\quad\quad\leq 8\epsilon^{2H} u^{2H} H K_{221}+16 \epsilon^{2H} u^2 K_{222}.
\end{eqnarray*}
where the last inequality is obtained by the same arguments of
   $ \mathbb{E}|I_1|^2 $, and $ K_{2i}, i=1,2 $ denote positive constants that may differ in different cases.
   Then
\begin{eqnarray*}
 & & \mathbb{E} |I_{2}|^2  \leq 4 \epsilon^{2H} H u^{2H-1} K_{211}\int_0^u \mathbb{E}(|X_\epsilon(s_1)-Z_\epsilon(s_1)|^2) du + \\
& & \quad\quad\quad\quad 16 \epsilon^{2H} u^2 K_{212}+8 \epsilon^{2H} u^{2H} H K_{221}+16 \epsilon^{2H} u^2 K_{222}.\quad\quad\quad\quad\quad\quad (Z_2)
\end{eqnarray*}

Therefore from above discussions $(Z_1)$ and $(Z_2)$, we can get
\begin{eqnarray*}
& &  \mathbb{E} ( |X_\epsilon(t)-Z_\epsilon(t)|^2) \\
& & \leq  4\epsilon^{4H} u K_{11} \int_0^u \mathbb{E} (
|X_\epsilon(s_1)-Z_\epsilon(s_1)|^2 ) du+ 8\epsilon^{4H} u^2 K_{12}+\\
 & &
\quad 8 \epsilon^{2H} H u^{2H-1} K_{211}\int_0^u \mathbb{E}(|X_\epsilon(s_1)-Z_\epsilon(s_1)|^2) du+ 16 \epsilon^{2H} u^2 K_{212} +\\
 & &
\quad 8 \epsilon^{2H} u^{2H} H K_{221} + 16 \epsilon^{2H} u^2 K_{222}\\
& & \leq (4 \epsilon^{4H} u K_{11}+8\epsilon^{2H} H u^{2H-1}K_{211})\int_0^u
 \mathbb{E} ( |X_\epsilon(s_1)-Z_\epsilon(s_1)|^2) du + \\
 & & \quad 8 \epsilon^{4H} u^2 K_{12} + 16\epsilon^{2H} u^2 K_{212}+ 8\epsilon^{2H} u^{2H} H K_{221}+16 \epsilon^{2H} u^2 K_{222}.
\end{eqnarray*}

Now by the Gronwall-Bellman inequality, we obtain
\begin{eqnarray*}
& & \mathbb{E} (
|X_\epsilon(t)-Z_\epsilon(t)|^2)\\
& & \leq 8 \epsilon^{2H} u( \epsilon^2 u K_{12}+ 2u K_{212}+ u^{2H-1} H  K_{221}+2u K_{222} )\\
& & \quad  \exp {4 \epsilon^{2H} (\epsilon^2 u K_{11}+2H u^{2H-1} K_{211}}).
\end{eqnarray*}

Select $ \beta \in (0, 1), L > 0 $ , such that for all $ t \in
(0,L\epsilon^{-2H \beta}]\subseteq [0,T] $ , we have
\begin{eqnarray*}
 \mathbb{E}( |X_\epsilon(t)-Z_\epsilon(t)|^2)  \leq K_3
\epsilon^{1-2H\beta}.
 \end{eqnarray*}
where
\begin{eqnarray*}
 & & K_3 = 8 L \epsilon^{1-2H\beta}(L\epsilon^{3-2H\beta-2H} K_{12}
  + 2L \epsilon^{1-2H\beta-2H}K_{212}+ \\
 & & \quad\quad\quad L^{2H-1}\epsilon^{2H\beta - 4 H^2\beta-4H^2+4H-1}H K_{221}+2L \epsilon^{1-2H \beta-2H}K_{222} )\\
 & & \quad\quad\quad \exp{4 \epsilon^{2H}(L\epsilon^{3-2 H\beta-2H} K_{11} + L^{2H-1} H \epsilon^{-4H^2\beta-2H^2+2H\beta +4H-1}K_{211})}.
\end{eqnarray*}
 it is a constant.

 Consequently, given any number $ \delta_1 > 0 $ , we can select
$ \epsilon_1 \in (0,\epsilon_0] $ , such that for every $ \epsilon \in
(0,\epsilon_1] $ , and for each $ t \in (0,L\epsilon^{-2H\beta}] $
\begin{eqnarray*}
 \mathbb{E}( |X_\epsilon(t)-Z_\epsilon(t)|^2)  \leq \delta_1.
\end{eqnarray*}
This is all of the proof.
\end{proof}

\medskip

We also have the following result on uniform convergence in probability.

\begin{theorem}
 Suppose that all assumptions (i)--(ii) and (C1)--(C2) are satisfied. Then for any number $ \delta_2 > 0 $ , we have
\begin{equation*}
   \lim_{\epsilon \rightarrow 0} \mathbb{P} ( |X_\epsilon(t)-Z_\epsilon(t)|>\delta_2) = 0,
\end{equation*}
where $ L $ and $ \beta $ are the same to Theorem 1.
\end{theorem}
\begin{proof}

On the basis of \textit{Theorem 1} and the Chebyshev-Markov inequality, for any
given number $ \delta_2 > 0 $ , one can find
\begin{equation*}
\mathbb{P} (|X_\epsilon(t)-Z_\epsilon(t)| > \delta_2)
\end{equation*}
\begin{equation*}
\quad\quad\quad\quad\quad\quad \leq \frac{1}{\delta_2^2}\mathbb{E}( |X_\epsilon(t)-Z_\epsilon(t)|^2)\leq
\frac{K_3}{\delta_2^2} \epsilon^{1-2H\beta}.
\end{equation*}
Let $ \epsilon \rightarrow 0 $ and the required result follows.
\end{proof}

\begin{remark}  {\bf\textit{Theorem 2}} means the convergence in
probability between the original solution  $ X_\epsilon(t) $ and the averaged solution $ Z_\epsilon(t) $ .
\end{remark}

Then, we also can study the forward integral and backward integral of stochastic differential equations driven by fBm, and the definition of the forward integral and backward integral are the same to section $2$ :
\begin{equation}
dX(t) = b(t,X(t))dt + \sigma(t,X(t))d^- B^H(t),
\end{equation}
\begin{equation}
dX(t) = b(t,X(t))dt + \sigma(t,X(t))d^+ B^H(t).
\end{equation}

On the basis of the Eq.(7) and Eq.(8), we can get the standard stochastic differential equation and the averaged SDEs :
\begin{equation*}
    X_\epsilon(t) = X(0) + \epsilon^{2H} \int_0^t b(s, X_\epsilon(s)) ds
    + \epsilon^{H}  \int_0^t \sigma(s, X_\epsilon(s))d^- B^H(s), \quad\quad\quad (12a)
\end{equation*}
\begin{equation*}
    Z_\epsilon(t) = X(0) + \epsilon^{2H} \int_0^t \bar{b}(Z_\epsilon(s)) ds + \epsilon^{H} \int_0^t \bar{\sigma}
    (Z_\epsilon(s))
    d^- B^H(s),  \quad\quad\quad\quad (12b)
\end{equation*}
\begin{equation*}
    X_\epsilon(t) = X(0) + \epsilon^{2H} \int_0^t b(s, X_\epsilon(s)) ds
    + \epsilon^{H}  \int_0^t \sigma(s, X_\epsilon(s))d^+ B^H(s), \quad\quad\quad (13a)
\end{equation*}
\begin{equation*}
    Z_\epsilon(t) = X(0) + \epsilon^{2H} \int_0^t \bar{b}(Z_\epsilon(s)) ds + \epsilon^{H} \int_0^t \bar{\sigma}
    (Z_\epsilon(s))
    d^+ B^H(s).  \quad\quad\quad\quad (13b)
\end{equation*}
where $ X(0) = X_0 $ is the initial condition, and the coefficients satisfy the $(i)-(v)$ conditions.\\
\begin{theorem}
Assume the original SDEs $(12,14)$ and the averaged SDEs $(13,15)$ both satisfy the
  (i)-(v) and (C1)-(C2). For a given arbitrarily small
number $ \delta_3 > 0 $ , there exists $ Q > 0 $ , $
\epsilon_2 \in (0,\epsilon_0] $ and  $ \beta \in (0, 1) $ ,
such that for any $ \epsilon \in (0,\epsilon_2] $ ,
\begin{equation*}
     \mathbb{E}(|X_\epsilon(t)- Z_\epsilon(t)|^2)\leq \delta_3,
\end{equation*}
And then for any number $ \delta_4>0 $, we can get
\begin{equation*}
   \lim_{\epsilon \rightarrow 0} \mathbb{P} ( |X_\epsilon(t)-Z_\epsilon(t)|>\delta_4) = 0.
\end{equation*}
\end{theorem}

\begin{proof}\\
Due to the Theorem $1$ and Theorem $2$, this proof is the similar to the process of SDEs with the symmetric integral.\\
We regard the forward integral of SDEs as an example.
\begin{eqnarray*}
 X_\epsilon(t)- Z_\epsilon(t) = \epsilon^{2H} \int_0^t
[b(s,X_\epsilon(s)) -\bar{b}(Z_\epsilon(s))] ds  + \\
\quad \quad\quad \epsilon^{H} \int_0^t [\sigma
    (s,X_\epsilon(s)) - \bar{\sigma}
    (Z_\epsilon(s))]
    d^{-} B^H(s),
\end{eqnarray*}
Then we can obtain
\begin{eqnarray*}
& & |X_\epsilon(t)- Z_\epsilon(t)|^2 \leq 2
 \epsilon^{4H} | \int_0^t
[b(s,X_\epsilon(s)) - \bar{b}(Z_\epsilon(s))] ds|^2 + \\
& &　\quad\quad\quad \quad\quad \quad\quad\quad 2 \epsilon^{2H} |\int_0^t [\sigma
    (s,X_\epsilon(s))-\bar{\sigma}
    (Z_\epsilon(s))]
    d^{-} B^H(s)|^2\\
& &　\quad\quad \quad\quad\quad\quad\quad  = F_1^2 + F_2^2,
\end{eqnarray*}
where  $0<t \leq u \in [0,T]$,$ F_i , i=1 , 2 $ denote the above terms respectively.
we get
\begin{eqnarray*}
& & F_1^2  \leq 4\epsilon^{4H} | \int_0^t
[b(s,X_\epsilon(s))-b(s,Z_\epsilon(s))] ds |^2 +\\
& &  \quad\quad 4\epsilon^{4H} |\int_0^t
[b(s,Z_\epsilon(s)) -\bar{b}(Z_\epsilon(s))] ds|^2\\
& & \quad  \leq F_{11}^2 + F_{12}^2,
\end{eqnarray*}
The similar technique yields
\begin{eqnarray*}
 |F_{11}|^2 \leq 4\epsilon^{4H}  t |\int_0^t
[b(s,X_\epsilon(s))- b(s,Z_\epsilon(s))]^2 ds|,
\end{eqnarray*}
\begin{eqnarray*}
 & &  \mathbb{E}|F_{11}|^2 \leq 4\epsilon^{4H} \mathbb{E} ( t \int_0^t
|b(s,X_\epsilon(s))- b(s,Z_\epsilon(s))|^2ds )\\
& &  \quad \quad \quad  \leq 4\epsilon^{4H} u L_{11} \int_0^u
\mathbb{E} (
|X_\epsilon(s_1)-Z_\epsilon(s_1)|^2 ) du,
\end{eqnarray*}
where $ L_{11} $ is a constant.
\begin{eqnarray*}
 & & \quad\mathbb{E} |F_{12}|^2 \leq  4\epsilon^{4H} \mathbb{E}( t^2
[\frac{1}{t}\int_0^t
|b(s,Z_\epsilon(s)) -\bar{b}(Z_\epsilon(s))|ds]^2)\\
& &  \quad\quad\quad\quad \leq 8\epsilon^{4H} u^2 L_{12},
\end{eqnarray*}
where $ L_{12} $ denotes a constant which may differ from $L_{11}$.
\\
We obtain
\begin{eqnarray*}
 \mathbb{E}|F_{1}|^2 \leq 4\epsilon^{4H} u L_{11} \int_0^u
\mathbb{E} (|X_\epsilon(s_1)-Z_\epsilon(s_1)|^2 ) du+ 8\epsilon^{4H} u^2 L_{12}.       \quad\quad        (Y_1)
\end{eqnarray*}
Consider the $ F_{2}$,
\begin{eqnarray*}
 & & \mathbb{E} |F_{2}|^2 \leq 4 \epsilon^{2H} \mathbb{E} (|\int_0^t [\sigma
    (s,X_\epsilon(s))-\sigma
    (s,Z_\epsilon(s))]d^\circ B^H(s)|^2)+\\
 & & \quad \quad \quad  \quad 4 \epsilon^{2H} \mathbb{E} (|\int_0^t [\sigma
    (s,Z_\epsilon(s))-\bar{\sigma}
    (Z_\epsilon(s))]d^\circ B^H(s)|^2) \\
& & \quad\quad\quad = F_{21}^2+F_{22}^2,
\end{eqnarray*}
By the previous conditions, we can get
\begin{eqnarray*}
 & & \quad F_{21}^2 \leq 4 \epsilon^{2H} \mathbb{E}(|\int_0^t
 M |X_\epsilon(s)-Z_\epsilon(s)|d^{-} B^H(s)|^2)\\
 & & \quad\quad\quad \leq 8 \epsilon^{2H} H u^{2H-1} L_{211}\int_0^u \mathbb{E}(|X_\epsilon(s_1)-Z_\epsilon(s_1)|^2)du + 16\epsilon^{2H} u^2 L_{212},
\end{eqnarray*}
\begin{eqnarray*}
 & & \quad F_{22}^2\leq 4 \epsilon^{2H} \mathbb{E}(|\int_0^u
 [\sigma(s,Z_\epsilon(s))-\bar{\sigma}(Z_\epsilon(s))]d^{-} B^H(s)|^2)\\
 & & \quad\quad\quad\leq 8\epsilon^{2H} u^{2H} H L_{221}+16 \epsilon^{2H} u^2 L_{222}.
\end{eqnarray*}
Therefore
\begin{eqnarray*}
 & & \mathbb{E} |F_{2}|^2  \leq 4 \epsilon^{2H} H u^{2H-1} L_{211}\int_0^u \mathbb{E}(|X_\epsilon(s_1)-Z_\epsilon(s_1)|^2) du + \\
& & \quad\quad\quad\quad 16 \epsilon^{2H} u^2 L_{212}+8 \epsilon^{2H} u^{2H} H L_{221}+16 \epsilon^{2H} u^2 L_{222}.\quad\quad\quad\quad\quad\quad (Y_2)
\end{eqnarray*}
Considering $(Y_{1})$ and $(Y_{2})$, one arrives at
\begin{eqnarray*}
& &  \mathbb{E} ( |X_\epsilon(t)-Z_\epsilon(t)|^2) \leq (4 \epsilon^{4H} u L_{11}+8\epsilon^{2H} H u^{2H-1}L_{211})\int_0^u
 \mathbb{E} ( |X_\epsilon(s_1)-Z_\epsilon(s_1)|^2) du + \\
 & & \quad 8 \epsilon^{4H} u^2 L_{12} + 16\epsilon^{2H} u^2 L_{212}+ 8\epsilon^{2H} u^{2H} H L_{221}+16 \epsilon^{2H} u^2 L_{222}.
\end{eqnarray*}

The discussions that follow are same to the process of proofs to  Theorem $1$ and Theorem $2$.
\end{proof}
\begin{remark} That is to say, by Theorem $1$, Theorem $2$ and Theorem $3$, we can get the same results for three types of pathwise integrals of SDEs.
\end{remark}

\section{\bf Examples}

Through the above discussion, we have established an averaging principle for the SDEs (6) with fractional Brownian motion. For Eq.(7) we
can define the standard SDEs and the averaged SDEs respectively
\begin{equation}
dX_\epsilon(t) = \epsilon^{2H} b(t, X_\epsilon(t)) dt +
\epsilon^H \sigma(t, X_\epsilon(t))d^\circ B^H(t),
\end{equation}
\begin{equation}
dZ_\epsilon(t) = \epsilon^{2H} \bar{b}(Z_\epsilon(t)) dt +
\epsilon^H \bar{\sigma}(Z_\epsilon(t))d^\circ B^H(t).
\end{equation}
with the same initial condition\\
\begin{eqnarray*}
  X_\epsilon(0) = Z_\epsilon(0) = X_0.
 \end{eqnarray*}

Assume that the conditions of {\bf\textit{Theorem 1}} are satisfied for $ b, \sigma $ , and
the similar conditions (C1)-(C2) are satisfied for
$ \bar{b}, \bar{\sigma} $. Then the following averaging principle
holds
\begin{eqnarray*}
& & \mathbb{E}(
|X_\epsilon(t)-Z_\epsilon(t)|^2)  \leq \delta_1 ,\\
& & \lim_{\epsilon\rightarrow 0} \mathbb{P} ( |X_\epsilon(t)-Z_\epsilon(t)|>\delta_2)=0.
 \end{eqnarray*}
where the  constants $ L , \beta, \epsilon, \delta_1 , \delta_2 $ are
the same as in {\bf\textit{Theorem 1 and Theorem 2}}.

Now we present two examples to demonstrate the procedure of the averaging
principle.
\begin{example}

 Consider the following SDEs driven by fractional Brownian motion :
\begin{equation} \label{Ex1}
dX_\epsilon = - 2\epsilon^{2H} \lambda X_\epsilon \sin^2(t) dt + \epsilon^H  d^\circ B^H(t),
\end{equation}
with initial condition $ X_\epsilon(0)=X_0 $ and $
\mathbb{E}|X_0|^2<\infty$, where $ b(t, X_\epsilon)= -2 \lambda
X_\epsilon  \sin^2(t),  \sigma(t, X_\epsilon)=1 $, and $ \lambda $ is a
positive constant, $ B^H(t) $ is a fractional Brownian  motion.
Then
\begin{equation*}
  \bar{b}(X_\epsilon) = \frac{1}{\pi}\int_0^{\pi} b(t, X_\epsilon)dt = - \frac{1}{2} \lambda X_\epsilon, \;\;\;\;\; \bar{\sigma}(
X_\epsilon)=1,
\end{equation*}
and define a new averaged SDE
\begin{equation*}
dZ_\epsilon =\epsilon^{2H} \bar{b}(Z_\epsilon) dt +
\epsilon^{H}\bar{\sigma}(Z_\epsilon)d^\circ B^H(t),
\end{equation*}

namely,
\begin{equation} \label{Ex1A}
dZ_\epsilon = - \frac{1}{2}\epsilon^{2H} \lambda Z_\epsilon dt + \epsilon^{H}
d^\circ B^H(t),
\end{equation}
Obviously, $ Z_\epsilon(t) $ is the well-known Ornstein-Uhlenbeck
process, and the solution can be obtained as :
\begin{equation*}
Z_\epsilon(t) =\exp(-\frac{1}{2}\epsilon \lambda t) X_0 + \sqrt{\epsilon}
\int_0^t \exp(-\frac{1}{2}\epsilon \lambda(t-s)) d^\circ B^H(t).
\end{equation*}

Because all the conditions (i)--(ii) and (C1)--(C2) are satisfied for functions
$ b, \sigma , \bar{b}, \bar{\sigma} $ in SDEs (6),(7), thus
{\bf\textit{Theorem 1}} and {\bf\textit{Theorem 2}} hold . That is,
\begin{eqnarray*}
 \mathbb{E}(|X_\epsilon(t)-Z_\epsilon(t)|^2)  \leq \delta_1,
 \end{eqnarray*}
and as $ \epsilon \rightarrow 0 $ ,
\begin{eqnarray*}
X_\epsilon(t) \rightarrow Z_\epsilon(t)  \quad \text{in
probability.}
 \end{eqnarray*}

\begin{figure}[th]
\begin{center}
\begin{minipage}[c]{0.5\textwidth}
\centerline{\psfig{file=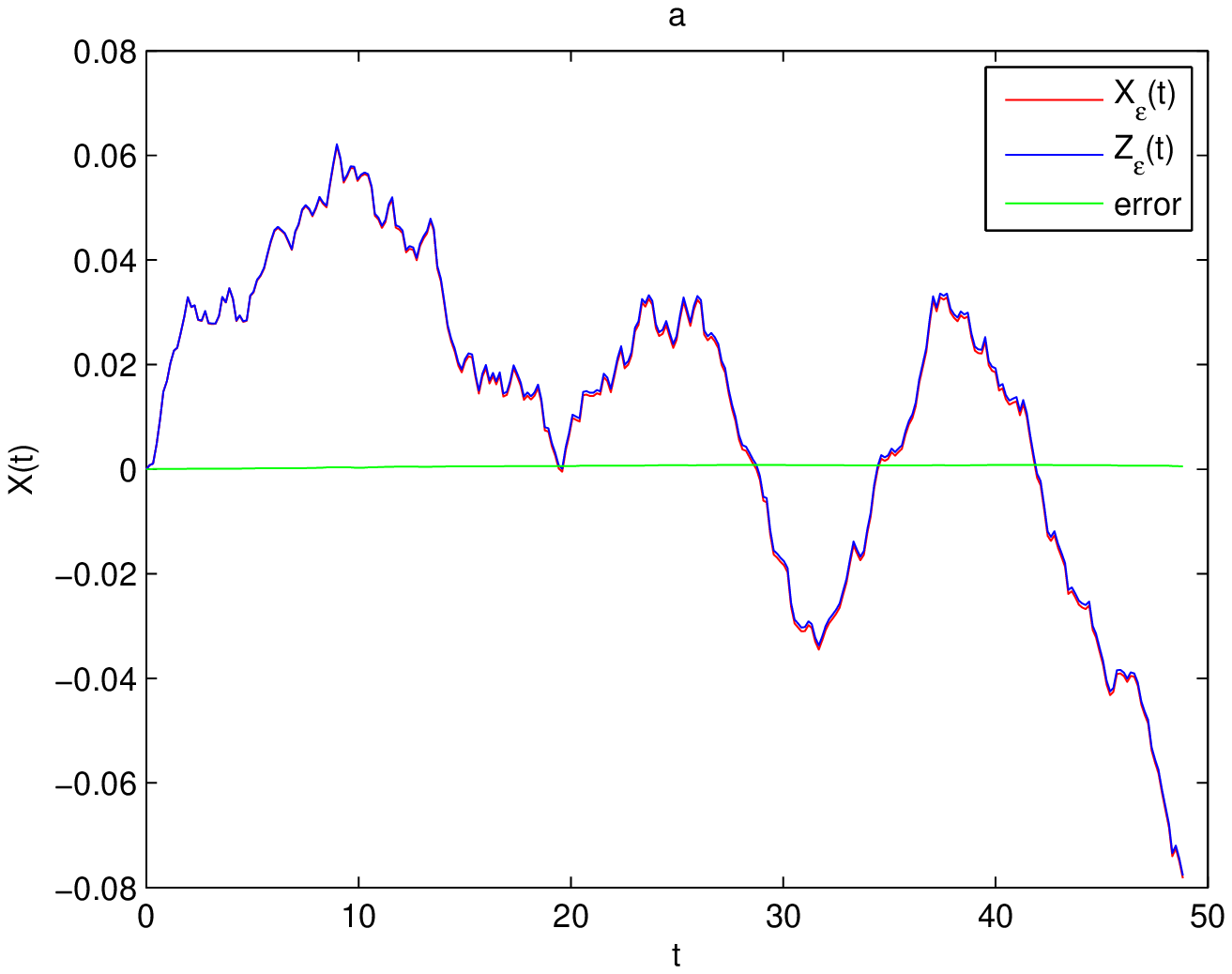,width=6.7cm}}
\vspace*{8pt}
\end{minipage}%
\begin{minipage}[c]{0.5\textwidth}
\centerline{\psfig{file=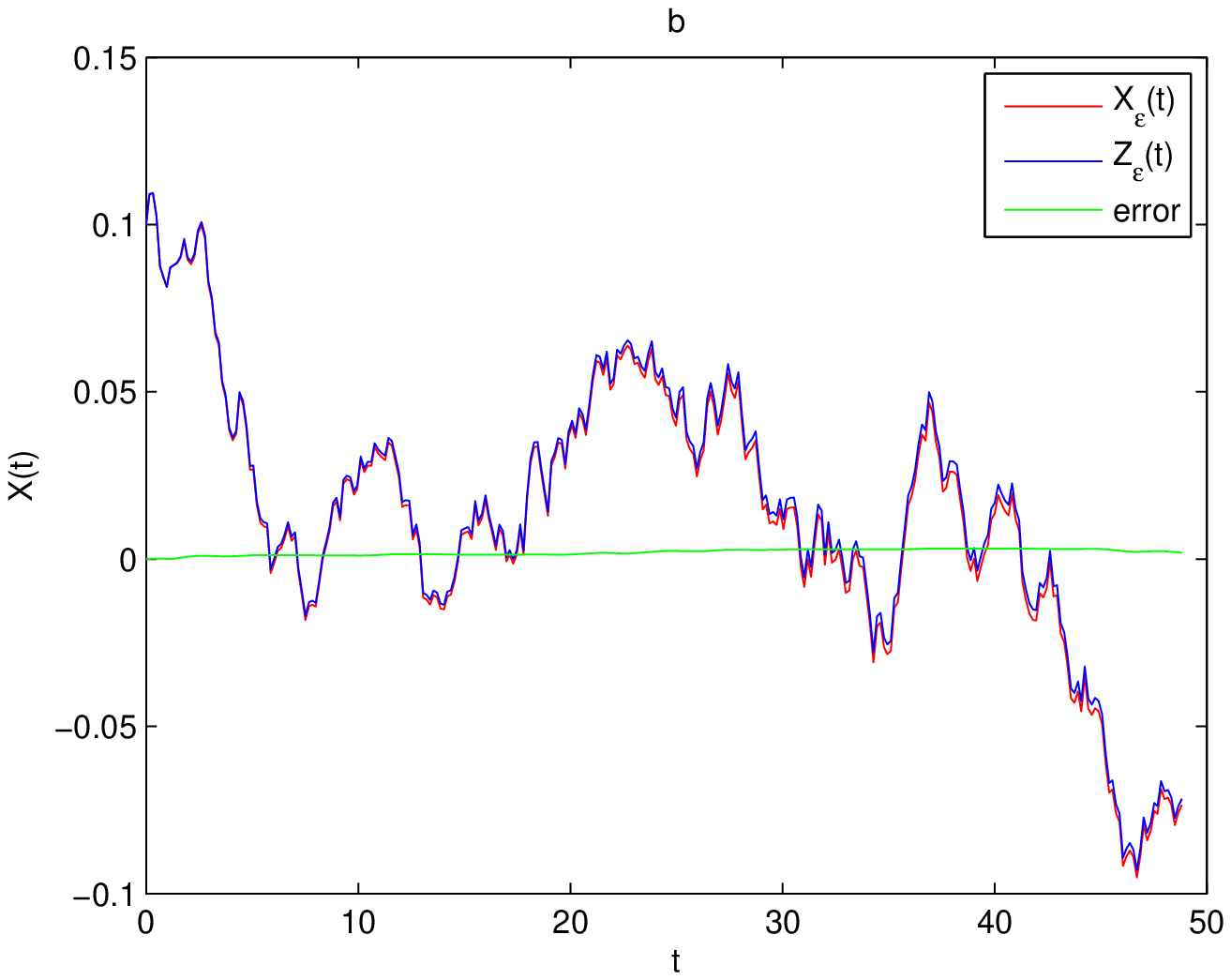,width=6.7cm}}
\end{minipage}
\begin{minipage}[c]{0.5\textwidth}
\centerline{\psfig{file=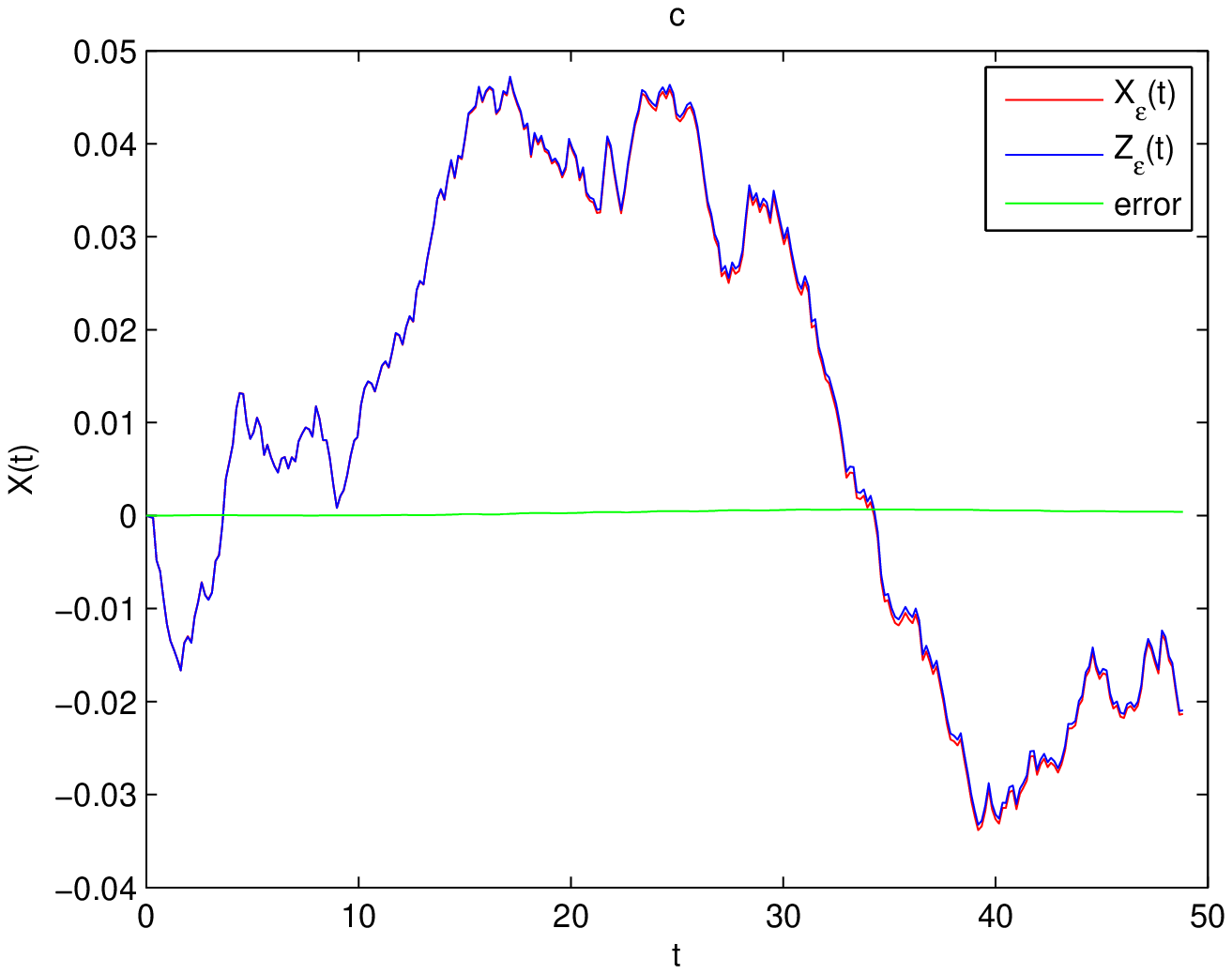,width=6.7cm}}
\end{minipage}%
\begin{minipage}[c]{0.5\textwidth}
\centerline{\psfig{file=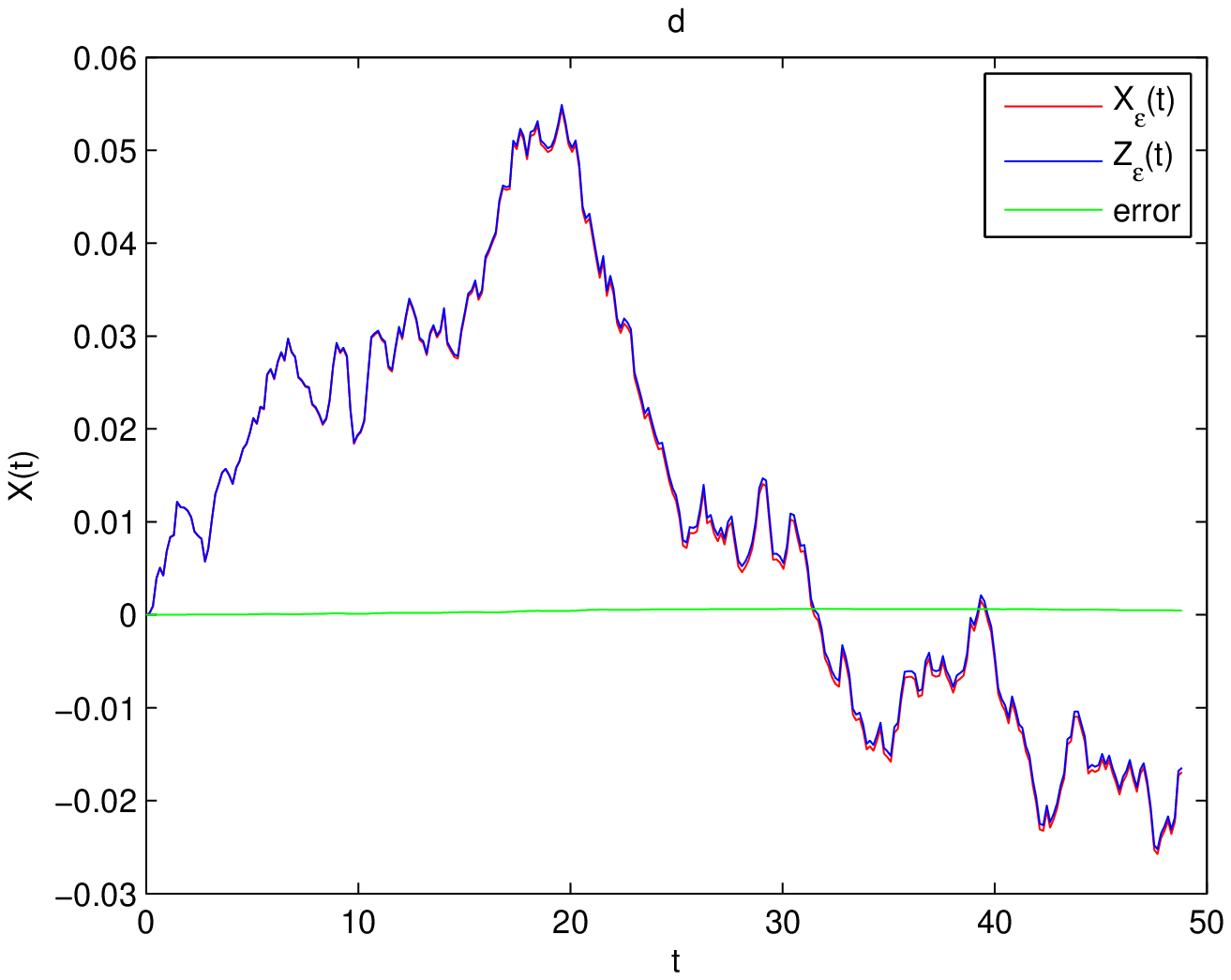,width=6.7cm}}
\end{minipage}
\renewcommand{\figurename}{Figure}
\caption{Comparison of the exact solution $X_\epsilon(t)$ with the
averaged solution $Z_\epsilon(t)$ for equations Eq.(14) and
Eq.(15)  (a) $X_0=0.0,\lambda=0.2,\epsilon=0.045,H=0.75 $ (b)
$X_0=0.1,\lambda=0.2,\epsilon=0.045,H=0.55$ (c) $X_0=0.1,\lambda=0.4,\epsilon=0.01,H=0.6 $ (d) $X_0=0.0,\lambda=0.4,\epsilon=0.02,H=0.7 $}  \label{Fig1}
\end{center}
\end{figure}
\end{example}

Now we carry out the numerical simulation to get the solutions of
\eqref{Ex1} and \eqref{Ex1A} under conditions of
$X_0=0.0,\lambda=0.2,\epsilon=0.045,H=0.75$ ,
$X_0=0.1,\lambda=0.2,\epsilon=0.045,H=0.85 $ ,$X_0=0.1,\lambda=0.4,\epsilon=0.01,H=0.6$,
$X_0=0.0,\lambda=0.4,\epsilon=0.02,H=0.7 $ respectively. Figure 1 $(a) \sim (d)$ show the comparison of exact solution $X_\epsilon(t)$ with
averaged solution $Z_\epsilon(t)$ for equations \eqref{Ex1} and
\eqref{Ex1A}. One can find a good agreement between solutions of
the original equation and the averaged equation.

\begin{example}

Consider the following SDE with the fractional Gaussian noise :
\begin{equation}\label{ex2}
dX_\epsilon = -\epsilon^{2H} dt + \epsilon ^H \cos^2 (t) \lambda d^\circ B^H(t),
\end{equation}

here we denote  $X(0)=X_0$ as the
initial condition with $\mathbb{E}|X(0)|^2<\infty$.

Here $b(t,X_\epsilon)=-1$, and $ \sigma(t, X_\epsilon)=\cos^2 (t) \lambda $.
Now we define a new (averaged) SDE as \\
\begin{equation}\label{ex2b}
dZ_\epsilon = -\epsilon^{2H} dt + \frac{3}{4} \lambda \epsilon^H d^\circ B^H(t).
\end{equation}

Obviously, all conditions in
{\bf \textit{Theorem 1 and Theorem 2}} are satisfied for the averaged SDE (17), so we can use
the solution $Z_\epsilon(t)$ to approximate the original solution
$X_\epsilon(t)$ to SDE (16), and the convergence in mean
square and in probability will be assured.

\begin{figure}[th]
\begin{center}
\begin{minipage}[c]{0.5\textwidth}
\centerline{\psfig{file=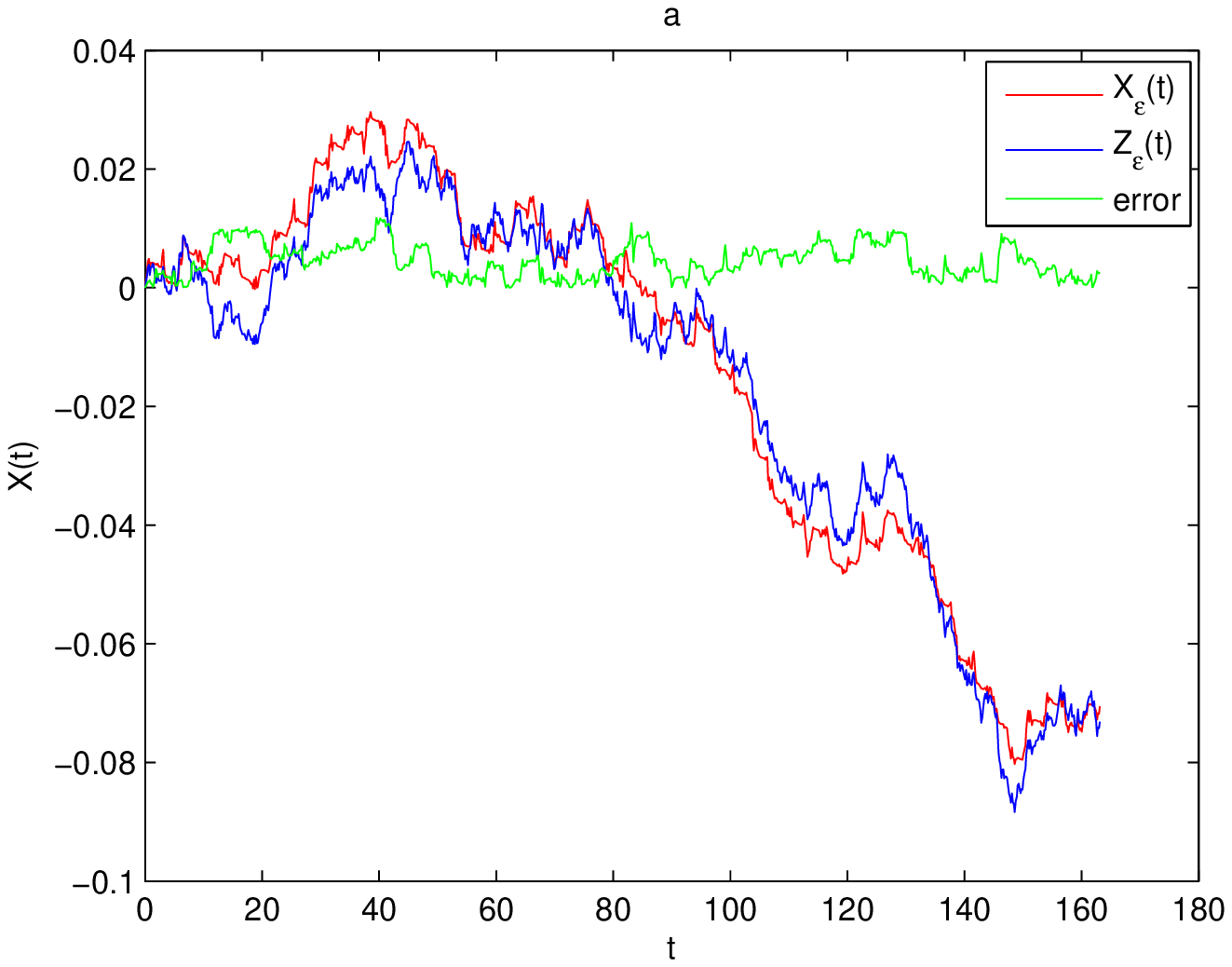,width=6.7cm}}
\vspace*{8pt}
\end{minipage}%
\begin{minipage}[c]{0.5\textwidth}
\centerline{\psfig{file=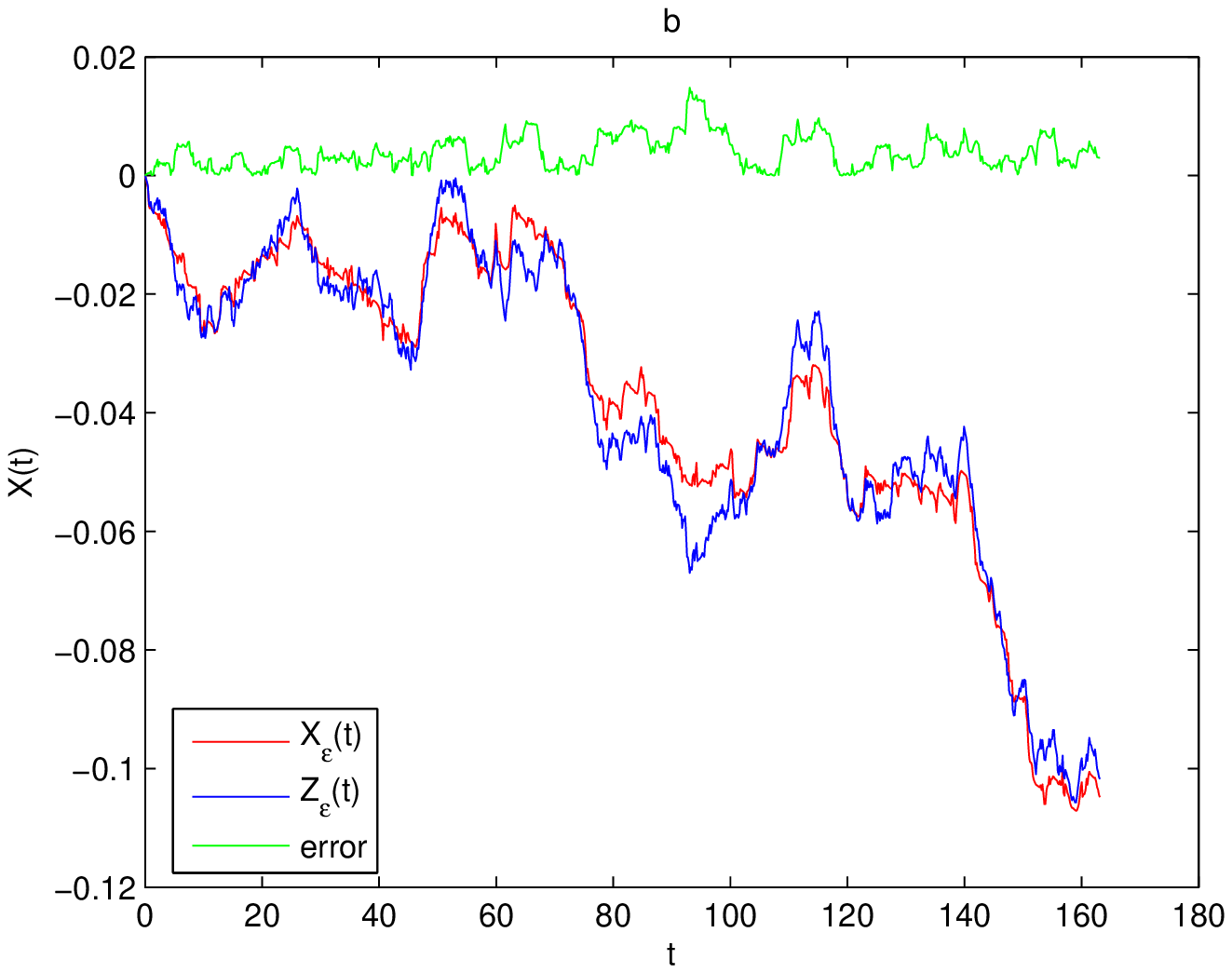,width=6.7cm}}
\end{minipage}
\begin{minipage}[c]{0.5\textwidth}
\centerline{\psfig{file=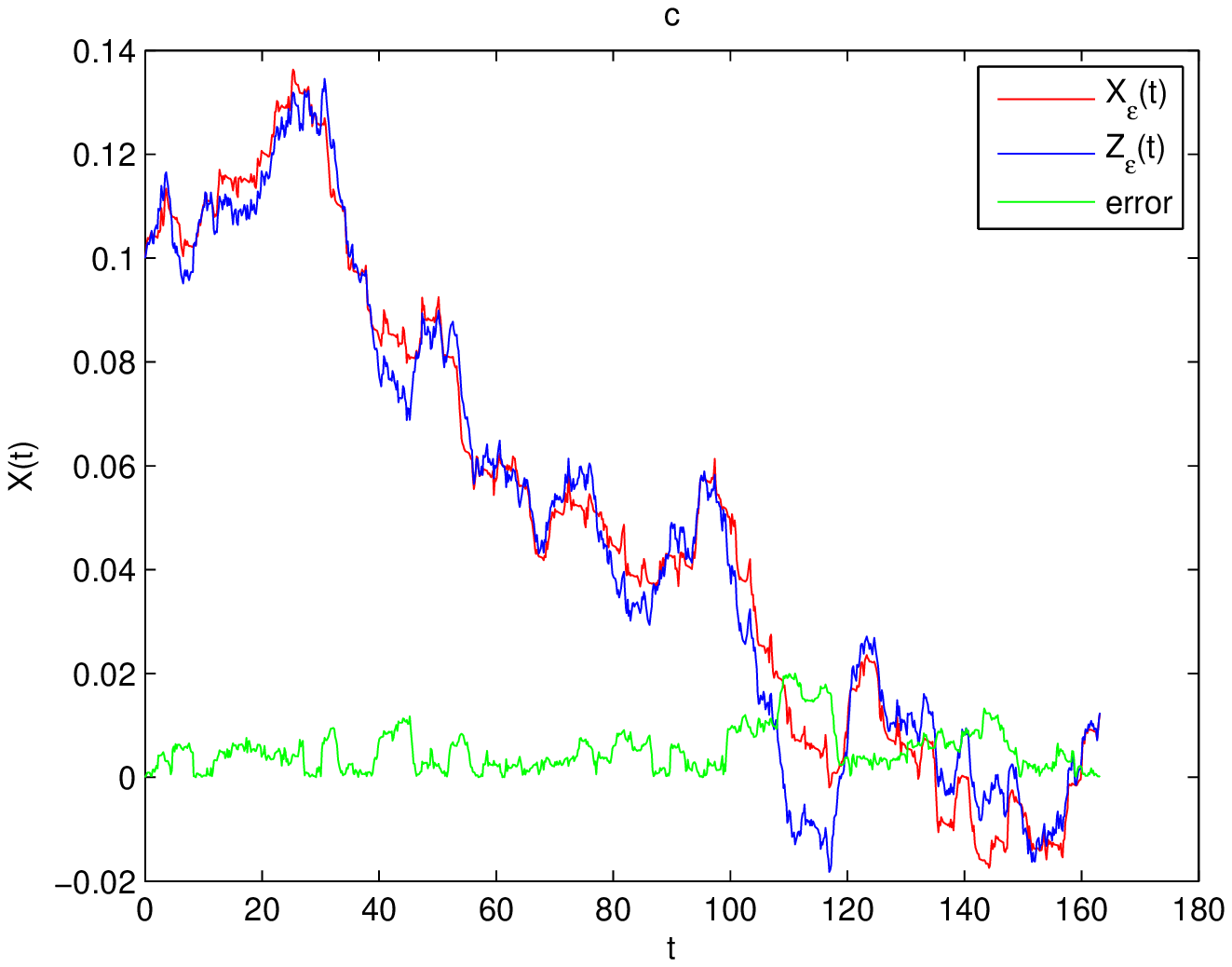,width=6.7cm}}
\end{minipage}%
\begin{minipage}[c]{0.5\textwidth}
\centerline{\psfig{file=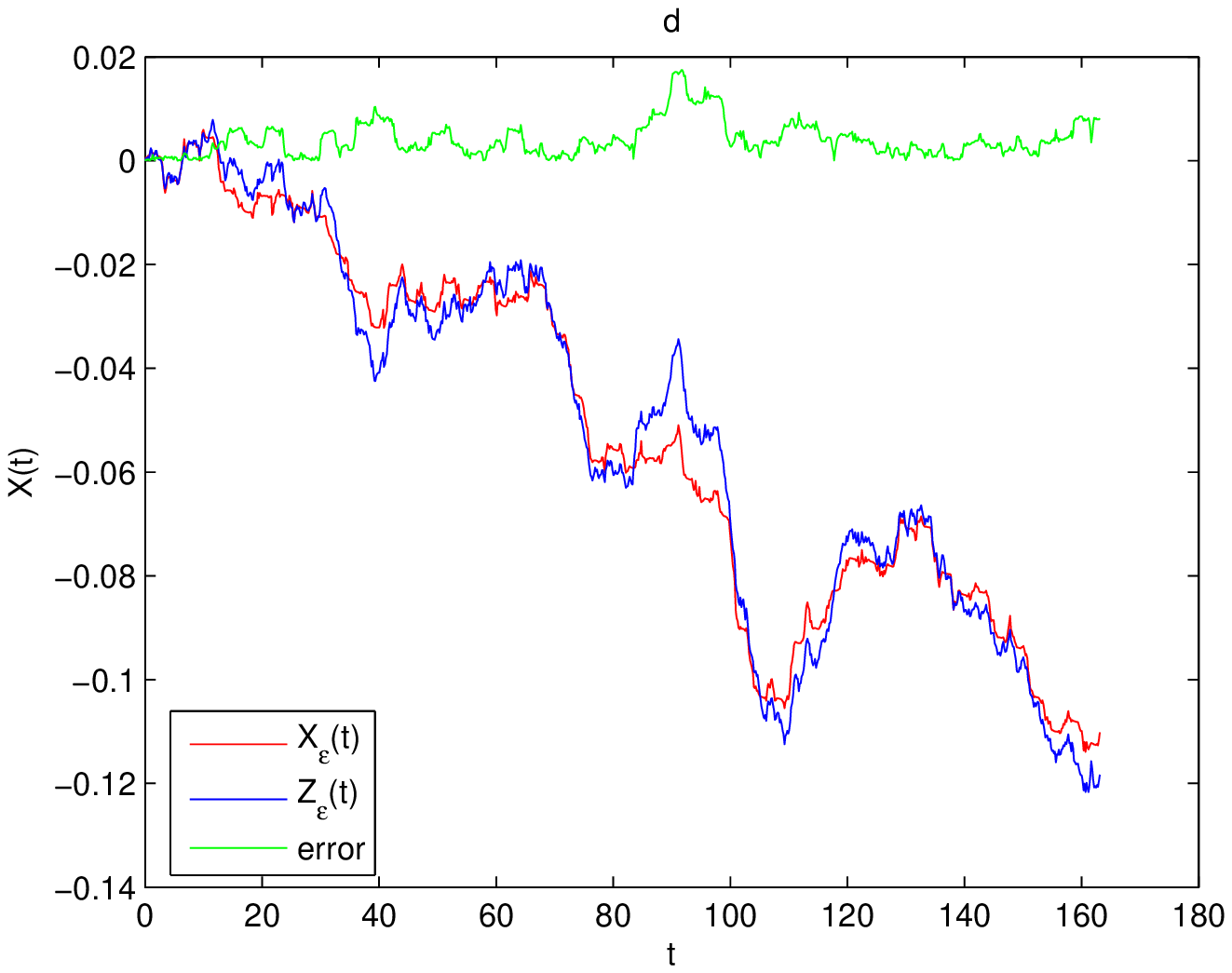,width=6.7cm}}
\end{minipage}
\renewcommand{\figurename}{Figure}
\caption{Comparison of the exact solution $X_\epsilon(t)$ with the
averaged solution $Z_\epsilon(t)$ for equations(17).(a) $X_0=0.0,\lambda=2.0,\epsilon=0.001,H=0.55 $ (b)
$X_0=0.0,\lambda=2.0,\epsilon=0.0045,H=0.65$ (c)$X_0=0.1,\lambda=3.0,\epsilon=0.002,H=0.6$ (d)$X_0=0.0,\lambda=3.0,\epsilon=0.002,H=0.7$}\label{Fig2}
 \label{Figure 2}
\end{center}
\end{figure}
\end{example}

\noindent  {\bf\large Acknowledgments}\\

This work was supported by the NSF of China (Grant Nos. 10972181, 11102157),
Program for NCET, the Shaanxi Project for Young New Star in Science and
Technology, NPU Foundation for Fundamental Research and SRF for ROCS, SEM.
We thank   Ilya Pavlyukevich for   valuable discussions, and referees for their helpful comments.

\end{document}